\documentclass[11pt,a4paper,subeqn]{amsart}

\usepackage[english]{babel}
\usepackage[T1]{fontenc}
\usepackage[utf8]{inputenc}
\usepackage{amsmath}
\usepackage{amssymb}
\usepackage{amsthm}
\usepackage{latexsym}
\usepackage{amsfonts}

\usepackage[onehalfspacing]{setspace} 
\setdisplayskipstretch{.421} 
\newlength{\abovebis} 
\newlength{\belowbis} 
\newlength{\aboveshortbis} 
\newlength{\belowshortbis} 
\everydisplay\expandafter{%
  \the\everydisplay 
  \advance\abovedisplayskip\abovebis 
  \advance\belowdisplayskip\belowbis 
  \advance\abovedisplayshortskip\aboveshortbis 
  \advance\belowdisplayshortskip\belowshortbis 
} 

\allowdisplaybreaks[3]

\def\R{\mathbb{R}}

\def\N{\mathbb{N}}
\def\C{\mathbb{C}}
\def\matn{M_{n}(\C)}

\def\Ree{\mathrm{Re}}
\def\Imm{\mathrm{Im}}

\theoremstyle{plain}

\newtheorem{lem}{Lemma}[section]
\newtheorem{theo}[lem]{Theorem}
\newtheorem{prop}[lem]{Proposition}

\newtheorem{cor}[lem]{Corollary}

\theoremstyle{definition}

\newtheorem{rem}{Remark}

\numberwithin{equation}{section}

\begin{document}
\title[Global uniqueness and reconstruction in 2D]{Global uniqueness and reconstruction for the multi-channel Gel'fand-Calder\'on inverse problem in two dimensions}
\author{Roman G. Novikov}
\author{Matteo Santacesaria}
\address[R. G. Novikov and M. Santacesaria]{CNRS (UMR 7641), Centre de Mathématiques Appliquées, \'Ecole Polytechnique, 91128, Palaiseau, France}
\email{novikov@cmap.polytechnique.fr, santacesaria@cmap.polytechnique.fr}
\begin{abstract}
We study the multi-channel Gel'fand-Calder\'on inverse problem in two dimensions, i.e. the inverse boundary value problem for the equation $-\Delta \psi + v(x) \psi = 0$, $x\in D$, where $v$ is a smooth matrix-valued potential defined on a bounded planar domain $D$. We give an exact global reconstruction method for finding $v$ from the associated Dirichlet-to-Neumann operator. This also yields a global uniqueness results: if two smooth matrix-valued potentials defined on a bounded planar domain have the same Dirichlet-to-Neumann operator then they coincide.
\end{abstract}

\maketitle

\section{Introduction}

Let $D$ be an open bounded domain in $\R^2$ with with $C^2$ boundary and let $v \in C^1(\bar D, M_{n}(\C))$, where $\matn$ is the set of the $n \times n$ complex-valued matrices. The Dirichlet-to-Neumann map associated to $v$ is the operator $\Phi : C^1(\partial D, \matn) \to L^p(\partial D, \matn), \; p < \infty$ defined by:
\begin{equation}
\Phi(f) = \left.\frac{\partial \psi}{\partial \nu}\right|_{\partial D}
\end{equation}
where $f \in C^1(\partial D, \matn)$, $\nu$ is the outer normal of $\partial D$ and $\psi$ is the $H^1 (\bar D, \matn)$-solution of the Dirichlet problem
\begin{equation} \label{equa}
-\Delta \psi + v(x) \psi =0 \; \; \textrm{on } D, \; \; \psi|_{\partial D} = f;
\end{equation}
here we assume that 
\begin{equation} \label{direig}
0 \textrm{ is not a Dirichlet eigenvalue for the operator } - \Delta + v \textrm{ in } D.
\end{equation}
Equation \eqref{equa} arises, in particular, in quantum mechanics, acoustics, electrodynamics; formally, it looks like the Schr\"odinger equation with potential $v$ at zero energy.

In addition, \eqref{equa} comes up as a 2D-approximation for the 3D equation (see section \ref{sec3d}).
\smallskip

The following inverse boundary value problem arises from this construction.
\smallskip

{\bf Problem 1.} Given $\Phi$, find $v$.\smallskip

This problem can be considered as the Gel’fand inverse boundary value problem for the multi-channel 2D Schr\"odinger equation at zero energy (see \cite{G}, \cite{N1}) and can also be seen as a generalization of the 2D Calder\'on problem for the electrical impedance tomography (see \cite{C}, \cite{N1}). In addition, the history of inverse problems for the two-dimensional Schr\"odinger equation at fixed energy goes back to \cite{DKN} (see also \cite{N3}, \cite{Gr} and references therein).
Note also that Problem 1 can be considered as a model problem for the monochromatic ocean tomography (e.g. see \cite{Bur} for similar problems arising in this tomography).

In the case of complex-valued potentials the global injectivity of the map $v \to \Phi$ was firstly proved in \cite{N1} for $D \subset \R^d$ with $d \geq 3$ and in \cite{B} for $d = 2$ with $v \in L^p$: in particular, these results were obtained by the use of global reconstructions developed in the same papers.

This is the first paper which gives global (uniqueness and reconstruction) results for Problem 1 with $\matn$-valued potentials with $n \ge 2$. Results in this direction were only known for potentials with many restrictions (e.g. see \cite{X}).

We emphasize that Problem 1 is not overdetermined, in the sense that we consider the reconstruction of a $\matn$-valued function $v(x)$ of two variables, $x \in D \subset \R^2$, from a $\matn$-valued function $\Phi(\theta,\theta')$ of two variables, $(\theta, \theta') \in \partial D \times \partial D$, where $\Phi(\theta,\theta')$ is the Schwartz kernel of the Dirichlet-to-Neumann operator $\Phi$: this is one of the principal differences between Problem 1 and its analogue for $D \subset \R^d$ with $d \ge 3$. At present, very few global results are proved for non-overdetermined inverse problems for the Schr\"odinger equation on $D \subset \R^d$ with $d \geq 2$. Concerning these results, our paper develops the two-dimensional works \cite{B}, \cite{NS} and indicates 3D applications of the method. The non-overdetermined inverse problems, including multi-channel ones, are much more developed for the Schr\"odinger equation in dimension $d =1$ (e.g. see \cite{AM}, \cite{ZS}).

We recall that in global results one does not assume that the potential $v$ is small in some sense or is (piecewise) real analytic or is subject to some other serious restrictions.
\smallskip

Our global reconstruction procedure for Problem 1 follows the same scheme as in the scalar case given in \cite{N1}, with some fundamental modifications inspired by \cite{B}.

Let us identify $\R^2$ with $\C$ and use the coordinates $z= x_1 + i x_2, \; \bar z = x_1 - i x_2$, where $(x_1, x_2) \in \R^2$. We define a special family of solutions of equation \eqref{equa}, which we call the Buckhgeim analogues of the Faddeev solutions: $\psi_{z_0}(z,\lambda)$, for $z, z_0 \in \bar D, \lambda \in \C$, such that $-\Delta \psi + v (x) \psi =0$ over $D$, where in particular $\psi_{z_0}(z,\lambda) \to e^{\lambda (z-z_0)^2}I$ for $\lambda \to \infty$ (i.e. for $|\lambda| \to + \infty$) and $I$ is the identity matrix.

More precisely, for a matrix valued potential $v$ of size $n$, we define $\psi_{z_0}(z,\lambda)$ as
\begin{equation}
\psi_{z_0}(z,\lambda) = e^{\lambda (z-z_0)^2}\mu_{z_0}(z,\lambda),
\end{equation}
where $\mu_{z_0}(\cdot,\lambda)$ solves the integral equation
\begin{equation} \label{defmu}
\mu_{z_0} (z,\lambda) = I+ \int_D g_{z_0}(z,\zeta,\lambda) v(\zeta) \mu_{z_0}(\zeta,\lambda) d \Ree \zeta \, d \Imm \zeta,
\end{equation}
$I$ is the identity matrix of size $n \in \N$, $z,z_0 \in \bar D, \; \lambda \in \C$ and
\begin{equation} \label{defg}
g_{z_0}(z,\zeta, \lambda) = \frac{e^{\lambda(\zeta-z_0)^2-\bar \lambda(\bar \zeta - \bar z_0)^2}}{4 \pi^2} \int_D \frac{e^{-\lambda(\eta -z_0)^2+\bar \lambda(\bar \eta -\bar z_0)^2}}{(z- \eta)(\bar \eta -\bar \zeta)} d\Ree \eta \, d \Imm \eta
\end{equation}
is a Green function of the operator $4\left( \frac{\partial}{\partial z} +2\lambda (z-z_0) \right) \frac{\partial}{\partial \bar z}$ in $D$, for $z_0 \in D$.
We consider equation \eqref{defmu}, at fixed $z_0$ and $\lambda$, as a linear integral equation for $\mu_{z_0}( \cdot, \lambda) \in C^1_{\bar z} (\bar D)$: we will see that it is uniquely solvable for $|\lambda| > \rho_1(D,N_1,n)$, where $\|v\|_{C^1_{\bar z}(\bar D, \matn)} < N_1$ (see Proposition \ref{prop3}).

In order to state the reconstruction method we also define the Bukhgeim analogue of the Faddeev generalized scattering amplitude
\begin{equation} \label{defh}
h_{z_0} (\lambda) = \int_D e^{\lambda (z-z_0)^2 - \bar \lambda (\bar z - \bar z_0)^2} v(z) \mu_{z_0} (z,\lambda) d \mathrm{Re}z \, d \mathrm{Im} z,
\end{equation}
for $z_0 \in \bar D, \; \lambda \in \C$.

\begin{theo} \label{theo2}
Let $D \subset \R^2$ be an open bounded domain with $C^2$ boundary and let $v \in C^1 (\bar D, \matn)$ be a matrix-valued potential which satisfies \eqref{direig} and $v|_{\partial D} = 0$. Consider, for $z_0 \in D$, the functions $h_{z_0}$, $\psi_{z_0}$, $g_{z_0}$ defined above and $\Phi, \Phi_0$ the Dirichlet-to-Neumann maps associated to the potentials $v$ and $0$, respectively. Then the following reconstruction formulas and equation hold:
\begin{align} \label{recv}
v(z_0) &= \lim_{\lambda \to \infty} \frac{2 }{\pi}|\lambda| h_{z_0}(\lambda), \\ \label{rech}
h_{z_0} (\lambda) &= \int_{\partial D} e^{-\bar \lambda (\bar z - \bar z_0)} (\Phi - \Phi_0) \psi_{z_0} (z,\lambda) |dz|,\\ \label{bordpsi}
\psi_{z_0}(z,\lambda)|_{\partial D}&= e^{\lambda(z-z_0)^2}I + \int_{\partial D}G_{z_0}(z,\zeta,\lambda)(\Phi-\Phi_0) \psi_{z_0}(\zeta,\lambda) |d \zeta|,
\end{align}
where
\begin{equation} \label{defG}
G_{z_0}(z,\zeta, \lambda) = e^{\lambda (z-z_0)^2}g_{z_0} (z,\zeta,\lambda)e^{-\lambda (\zeta-z_0)^2},
\end{equation}
$z_0 \in D$, $z, \zeta \in \partial D$, $\lambda \in \C$, $|\lambda| > \rho_1(D,N_1,n)$,  where $\|v\|_{C^1_{\bar z}(\bar D, \matn)} < N_1$.

In addition, if $v \in C^2(\bar D, \matn)$ with $\|v\|_{C^2(\bar D, \matn)} < N_2$ and $\frac{\partial v}{\partial \nu}|_{\partial D} = v|_{\partial D}= 0$ then the following estimates hold:
\begin{subequations} \label{estimv}
\begin{align} \label{esttv}
\left|v(z_0) - \frac{2}{\pi}|\lambda| h_{z_0}(\lambda)\right| &\leq a(D,n) \frac{\log(3|\lambda|)}{|\lambda|^{1/2}}N_2(N_2+1), \\
\label{esttv2}
\left|v(z_0) - \frac{2}{\pi}|\lambda| h_{z_0}(\lambda)\right| &\leq b(D,n) \frac{(\log(3|\lambda|))^2}{|\lambda|^{3/4}}N_2(N_2^2+1),
\end{align}
\end{subequations}
for $|\lambda| > \rho_2(D,N_1,n)$, $z_0 \in D$.
\end{theo}

\begin{rem}
Note that in Theorem \ref{theo2}, $\rho_j = \rho_j(D,N_1,n)$, $j=1,2$ (where $\|v\|_{C^1_{\bar z}(\bar D, \matn)} < N_1$), are arbitrary fixed positive constants such that
\begin{equation}
\begin{split}
2n\frac{c_2(D)}{ |\lambda|^{\frac 1 2}}\|v\|_{C^1_ {\overline z}(\bar D)} < 1,\ |\lambda|\ge 1, \textrm{ if} \; |\lambda| > \rho_1, \\
2n\frac{c_2(D)}{ |\lambda|^{\frac 1 2}}\|v\|_{C^1_ {\overline z}(\bar D)} \le \frac 1 2 ,\ |\lambda|\ge 1, \textrm{ if} \; |\lambda| > \rho_2,
\end{split}
\end{equation}
where $c_2$ is the constant in Lemma \ref{lem1}.
\end{rem}

\begin{rem}
Note that estimate \eqref{esttv2} is not strictly stronger than \eqref{esttv} because of the presence of the $N_2^3$ factor.
\end{rem}

In order to make use of the reconstruction given by Theorem \ref{theo2}, the following two propositions are necessary:

\begin{prop}\label{prop2}
Under the assumptions of Theorem \ref{theo2} (without the additional assumptions used for \eqref{estimv}), equation \eqref{bordpsi} is a Fredholm linear integral equation of the second kind for $\psi_{z_0} \in C(\partial D)$.
\end{prop}

\begin{prop} \label{prop3}
Under the assumptions of Theorem \ref{theo2} (without the additional assumptions used for \eqref{estimv}), for $|\lambda| > \rho_1(D,N_1,n)$, where $\|v\|_{C^1_{\bar z}(\bar D, \matn)} < N_1$, equations \eqref{defmu} and \eqref{bordpsi} are uniquely solvable in the spaces of continuous functions on $\bar D$ and $\partial D$, respectively.
\end{prop}

\begin{rem} \label{rem3}
Note that the assumption that $v|_{\partial D} = 0$ is unnecessary for formula \eqref{rech}, equation \eqref{bordpsi} and Propositions \ref{prop2}, \ref{prop3}. In addition, formula \eqref{recv} also holds without this assumption if
\begin{equation} \label{condbord}
\int_{\partial D} e^{\lambda (z-z_0)^2 - \bar \lambda (\bar z - \bar z_0)^2} w(z)|dz| \to 0 \qquad \text{as } |\lambda| \to \infty,
\end{equation}
for fixed $z_0 \in D$ and each $w \in C^1(\partial D)$. The class of domains $D$ for which \eqref{condbord} holds for each $z_0 \in D$ is large and includes, for example, all ellipses.

Note also that if $v|_{\partial D} \neq 0$ but $v \equiv \Lambda \in \matn$ on some open neighborhood of $\partial D$ in $\bar D$, then estimates \eqref{estimv} hold with $h_{z_0} (\lambda)$ replaced by
\begin{equation} \label{hplus}
h^+_{z_0}(\lambda) = h_{z_0}(\lambda) + \int_{\R^2 \setminus D} e^{\lambda(z-z_0)^2 - \bar \lambda (\bar z - \bar z_0)^2} \Lambda \chi(z) d \Ree z \, d \Imm z,
\end{equation}
where $\chi \in C^2(\R^2 , \R)$, $\chi \equiv 1$ on $D$, $\mathrm{supp} \chi$ is compact, and with the constants $a,b$ depending also on $\chi$.
The aforementioned matrix $\Lambda$, for example, can be related with a diagonal matrix composed by the eigenvalues $\{ \lambda_i \}_{1 \leq i \leq n}$ arising in section \ref{sec3d}.
\end{rem}

Theorem \ref{theo2} and Propositions \ref{prop2}, \ref{prop3} yield the following corollary:

\begin{cor} \label{theo1}
Let $D \subset \R^2$ be an open bounded domain with $C^2$ boundary, let $v_1, v_2 \in C^1(\bar D, \matn)$ be two matrix-valued potentials which satisfy \eqref{direig} and $\Phi_1, \Phi_2$ the corresponding Dirichlet-to-Neumann operators. If $\Phi_1 = \Phi_2$ then $v_1 = v_2$.
\end{cor}

Theorem \ref{theo2}, Propositions \ref{prop2}, \ref{prop3} and Corollary \ref{theo1} are proved in section \ref{secpf}.
\smallskip

The global reconstruction of Theorem \ref{theo2} is fine in the sense that is consists in solving Fredholm linear integral equations of the second type and using explicit formulas; nevertheless this reconstruction is not optimal with respect to its stability properties: see \cite{BRS}, \cite{N4}, \cite{BKM} for discussions and numerical implementations of the aforementioned similar (but overdetermined) reconstruction of \cite{N1} for $d=3$ and $n=1$. An approximate but more stable reconstruction method for Problem 1 will be published in another paper.

The present paper is focused on global uniqueness and reconstruction for Problem 1 for $n \geq 2$. In addition, using the techniques developed in the present work and following the  scheme of \cite{NS} it is also possible to obtain a global logarithmic stability estimate for Problem 1 in the multi-channel case. Following inverse problem traditions (e.g. see \cite{A}, \cite{N4}, \cite{NS}) this result will be published in another paper.
\smallskip

{\bf Acknowledgements.} We thank V. A. Burov, O. D. Rumyantseva, S. N. Sergeev for very useful discussions.

\section{Approximation of the 3D equation} \label{sec3d}

In this section we recall how the multi-channel two-dimensional Schr\"odin\-ger equation can be seen as an approximation of the scalar 3D equation in a cylindrical domain; in this framework, three-dimensional inverse problems can be approximated by two-dimensional ones.

Let $L= [a,b]$ for some $a,b \in \R$ and consider the complex-valued potential $v(x,z)$ defined on the set $D \times L$, where $x = (x_1, x_2) \in D \subset \R^2, \; z \in L$. We consider the equation
\begin{equation} \label{3d}
-\Delta \psi(x,z) + v(x,z) \psi(x,z) = 0 \quad \textrm{in } D \times L.
\end{equation}
Now, for every $x \in D$ we can write $\psi(x,z) = \sum_{j=1}^{\infty} \psi_j(x) \phi_j(z)$, where $\{\phi_j\}$ is the orthonormal basis of $L^2 (L)$ given by the eigenfunctions of $- \frac{d^2}{dz^2}$: more precisely
\begin{align} \label{cond1}
-\frac{d^2}{dz^2} \phi_j(z) &= \lambda_j \phi_j(z) \textrm{ for } z \in L, \\ \label{cond2}
\phi_j|_{\partial L} &= 0 \quad \textrm{(for example)} \\ \nonumber
\int_{L}\bar \phi_i(z) \phi_j(z) dz &= \delta_{ij} 
\end{align}
and $\psi_j(x) = \int_L \psi(x,z) \bar \phi_j(z) dz$. Now equation \eqref{3d} reads
\begin{align} \label{cond3}
\sum_{j=1}^{\infty} \left( -\Delta_x \psi_j(x) \phi_j(z) - \psi_j(x) \Delta_z \phi_j(z) \right) + v(x,z)\sum_{j=1}^{\infty} \psi_j(x) \phi_j(z) = 0.
\end{align}
Using \eqref{cond1}-\eqref{cond3} and the properties of $\{ \phi_j(z) \}$, we obtain that equation \eqref{3d} is equivalent to the following infinite-dimensional system
\begin{align}
-\Delta_x \psi_i(x) + \lambda_i \psi_i(x) + \sum_{j=1}^{\infty} V_{ij}(x) \psi_j(x) = 0, \textrm{ for } i = 1,\ldots,
\end{align}
where 
\begin{align*}
V_{ij}(x) = \int_L \bar \phi_i(z) v(x,z) \phi_j(z) dz.
\end{align*}
Notice that if $\bar v =v$ then $V^{\ast} = V$. Now, if we impose $1 \leq i,j \leq n$ for some $n \in \N$, we find equation \eqref{equa}.

We also give here the relation between the Dirichlet-to-Neumann (D-t-N) operators of the 3D equation and that of the 2D multi-channel equation. If $\Phi(\theta,z, \theta',z')$ is the Schwartz kernel of the D-t-N operator of the 3D problem, and $(\Phi_{ij}(\theta,\theta'))_{i,j \geq 1}$ that of the 2D infinity-channel problem, we have
\begin{equation} \label{diri3}
\Phi_{ij}(\theta,\theta') = \int_{L \times L} \Phi(\theta,z, \theta',z')\bar \phi_i(z) \phi_j(z')dz \, dz',
\end{equation}
where $\theta, \theta' \in \partial D$, $z,z' \in L$.
This follows from
\begin{align}
\int_{\partial D \times L} \Phi(\theta,z,\theta',z')f(\theta',z')d\theta' \, dz' = \sum_{i=1}^{\infty} \left( \sum_{j=1}^{\infty} \int_{\partial D} \Phi_{ij}(\theta, \theta')f_j(\theta') d\theta' \right) \phi_i(z),
\end{align}
for every $f \in C^{1}(\partial (D \times L))$ such that $f|_{D \times \partial L} = 0$ and $f(\theta,z) = \sum_{j=1}^{\infty} f_j(\theta) \phi_j(z)$.
\smallskip

Let us remark that reductions of 3D direct and inverse problems to multi-channel 2D problems are well known in the physical literature for a long time (e.g. see \cite{Bur}). Nevertheless, we do not know a reference containing formula \eqref{diri3} in its precise form.

\section{Preliminaries}

In this section we introduce and give details about the above-mentioned family of solutions of equation \eqref{equa}, which will be used throughout all the paper.
\smallskip

Let us define the function spaces $C^1_{\bar z}(\bar D) = \{ u : u, \frac{\partial u}{ \partial \bar z} \in C(\bar D, \matn) \}$ with the norm $\|u \|_{C^1_{\bar z}(\bar D)} = \max ( \|u\|_{C(\bar D)}, \| \frac{\partial u}{\partial \bar z} \|_{C(\bar D)} )$, $\|u\|_{C(\bar D)} = \sup_{z \in \bar D} |u|$ and $|u| = \max_{1 \leq i,j \leq n} |u_{i,j}|$; we define also $C^1_{z}(\bar D) = \{ u : u, \frac{\partial u}{ \partial z} \in C(\bar D, \matn) \}$ with an analogous norm.

The functions $G_{z_0}(z,\zeta, \lambda), \; g_{z_0}(z,\zeta, \lambda), \; \psi_{z_0}(z,\lambda), \; \mu_{z_0}(z,\lambda)$ defined in Section 1, satisfy

\begin{align} \label{eq21}
4 \frac{\partial^2}{\partial z \partial \bar z} &G_{z_0} (z, \zeta, \lambda) = \delta(z-\zeta), \\ \label{zetaaa}
4\frac{\partial^2}{\partial \zeta \partial \bar \zeta} &G_{z_0}(z,\zeta,\lambda) = \delta(\zeta-z),\\ \label{eq22}
4\left(\frac{\partial}{\partial z} + 2\lambda (z-z_0) \right) \frac{\partial}{\partial \bar z} &g_{z_0}(z,\zeta, \lambda) = \delta(z-\zeta), \\
4\frac{\partial}{\partial \bar \zeta} \left(\frac{\partial}{\partial \zeta} - 2\lambda (\zeta-z_0) \right) &g_{z_0}(z,\zeta, \lambda) = \delta(\zeta-z),\\ \label{eq23}
-4 \frac{\partial^2}{\partial z \partial \bar z} &\psi_{z_0} (z,\lambda) + v(z) \psi_{z_0}(z,\lambda) = 0,\\ \label{eq24}
-4 \left(\frac{\partial}{\partial z} + 2\lambda (z-z_0) \right) \frac{\partial}{\partial \bar z} &\mu_{z_0}(z, \lambda) + v(z) \mu_{z_0} (z,\lambda) = 0,
\end{align}
where $z, z_0 , \zeta \in D$, $\lambda \in \C$, $\delta$ is the Dirac's delta. (In addition, it is assumed that \eqref{defmu} is uniquely solvable for $\mu_{z_0}( \cdot, \lambda) \in C^1_{\bar z} (\bar D)$ at fixed $z_0$ and $\lambda$.) Formulas \eqref{eq21}-\eqref{eq24} follow from \eqref{defmu}, \eqref{defg}, \eqref{defG} and from
\begin{align*}
\frac{\partial}{\partial \bar z} &\frac{1}{\pi (z-\zeta)} = \delta(z-\zeta), \\
\left(\frac{\partial}{\partial z} + 2\lambda (z-z_0) \right) &\frac{e^{-\lambda(z-z_0)^2+\bar \lambda (\bar z- \bar z_0)^2}}{\pi (\bar z -\bar \zeta) } e^{\lambda(\zeta - z_0)^2- \bar \lambda(\bar \zeta -\bar z_0)^2}  = \delta(z-\zeta),
\end{align*}
where $z, \zeta, z_0 , \lambda \in \C$.

We say that the functions $G_{z_0}$, $g_{z_0}$, $\psi_{z_0}$, $\mu_{z_0}$, $h_{z_0}$ are the Bukhgeim-type analogues of the Faddeev functions (see \cite{NS}). We recall that the history of these functions goes back to \cite{F} and \cite{BC}.
\smallskip

Now we state some fundamental lemmata.
Let
\begin{equation} \label{green}
g_{z_0, \lambda} u(z) = \int_D g_{z_0} (z, \zeta, \lambda) u(\zeta) d \mathrm{Re}\zeta \, d \mathrm{Im}\zeta, \; z \in \bar D, \; z_0, \lambda \in \C,
\end{equation}
where $g_{z_0}(z, \zeta, \lambda)$ is defined by \eqref{defg} and $u$ is a test function.

\begin{lem}[\cite{NS}] \label{lem1}
Let $g_{z_0, \lambda} u$ be defined by \eqref{green}. Then, for $z_0, \lambda \in \C$, the following estimates hold:
\begin{align} \label{contg}
&g_{z_0, \lambda} u \in C^1_{\bar z}(\bar D), \quad \textrm{for} \; u \in C(\bar D), \\ \label{estcontg}
\| &g_{z_0, \lambda} u \|_{C^1(\bar D)} \leq c_1(D,\lambda) \|u\|_{C(\bar D)}, \quad \textrm{for} \; u \in C(\bar D), \\ \label{est1}
\| &g_{z_0, \lambda} u \|_{C^1_{\bar z}(\bar D)} \leq \frac{c_2(D)}{|\lambda|^{\frac 1 2}} \|u \|_{C^1_{\bar z}(\bar D)}, \quad \textrm{for} \; u \in C^1_{\bar z}(\bar D), \; |\lambda| \geq 1.
\end{align}
\end{lem}

Given a potential $v \in C^1_{\bar z}(\bar D)$ we define the operator $g_{z_0,\lambda} v$ simply as $(g_{z_0,\lambda} v) u(z) = g_{z_0,\lambda} w(z), \; w=vu$, for a test function $u$. If $u \in C^1_{\bar z}(\bar D)$, by Lemma \ref{lem1} we have that $g_{z_0,\lambda} v : C^1_{\bar z}(\bar D)  \to C^1_{\bar z}(\bar D)$,
\begin{equation} \label{estsolmu}
\| g_{z_0,\lambda} v \|^{op}_{C^1_{\bar z}(\bar D)} \leq 2n \| g_{z_0,\lambda} \|^{op}_{C^1_{\bar z}(\bar D)} \|v\|_{C^1_{\bar z}(\bar D)},
\end{equation}
where $\| \cdot \|^{op}_{C^1_{\bar z}(\bar D)}$ denotes the operator norm in $C^1_{\bar z}(\bar D)$, $z_0, \lambda \in \C$. In addition, $\| g_{z_0,\lambda} \|^{op}_{C^1_{\bar z}(\bar D)}$ is estimated in Lemma \ref{lem1}. Inequality \eqref{estsolmu} and Lemma \ref{lem1} implies existence and uniqueness of $\mu_{z_0}(z, \lambda)$ (and thus also $\psi_{z_0}(z,\lambda)$) for $|\lambda| > \rho_1(D,N_1,n)$.
\smallskip

Let
\begin{align*}
\mu^{(k)}_{z_0}(z, \lambda) &= \sum_{j=0}^k (g_{z_0,\lambda} v)^j I, \\
h^{(k)}_{z_0}(\lambda) &= \int_D e^{\lambda (z-z_0)^2 -\bar \lambda (\bar z- \bar z_0)^2}v(z) \mu^{(k)}_{z_0}(z,\lambda) d\Ree z \, d \Imm z,
\end{align*} 
where $z,z_0 \in D$, $\lambda \in \C$, $k \in \N \cup \{ 0 \}$.

\begin{lem}[\cite{NS}] \label{lem2}
For $v \in C^1_{\bar z}(\bar D)$ such that $v|_{\partial D} = 0$ the following formula holds:
\begin{equation} \label{estp}
v(z_0) = \frac{2}{\pi} \lim_{\lambda \to \infty} |\lambda| h^{(0)}_{z_0}(\lambda), \qquad z_0 \in D.
\end{equation}
In addition, if $v \in C^2(\bar D)$, $v|_{\partial D}= 0$ and $\frac{\partial v}{\partial \nu}|_{\partial D} = 0$ then
\begin{equation} \label{estm}
\left|v(z_0) - \frac{2}{\pi}|\lambda| h^{(0)}_{z_0}(\lambda) \right| \leq c_3(D,n) \frac{\log(3|\lambda|)}{|\lambda|}\|v\|_{C^2(\bar D)},
\end{equation}
for $z_0 \in D$, $\lambda \in \C$, $|\lambda| \geq 1$.
\end{lem}

Following the proof of \cite[Lemma 6.2]{NS} and assuming \eqref{condbord}, we have that limit \eqref{estp} is valid without the assumption that $v|_{\partial D} =0$.
In addition, if $v|_{\partial D} \neq 0$ but $v \equiv \Lambda \in \matn$ on some open neighborhood of $\partial D$ in $\bar D$, then estimate \eqref{estm} holds with $h^{(0)}_{z_0} (\lambda)$ replaced by
\begin{equation}
h^{(0),+}_{z_0}(\lambda) = h^{(0)}_{z_0}(\lambda) + \int_{\R^2 \setminus D} e^{\lambda(z-z_0)^2 - \bar \lambda (\bar z - \bar z_0)^2} \Lambda \chi(z) d \Ree z \, d \Imm z,
\end{equation}
where $\chi \in C^2(\R^2 , \R)$, $\chi \equiv 1$ on $D$, $\mathrm{supp} \chi$ is compact, and the constant $c_3$ depending also on $\chi$.\smallskip

Let
\begin{equation}
W_{z_0}(\lambda)=\int_D e^{\lambda(z-z_0)^2-\bar\lambda(\bar z-\bar z_0)^2}
w(z)d\Ree\,z d\Imm\,z,
\end{equation}
where $z_0\in\bar D$, $\lambda\in\C$ and $w$ is some $\matn$-valued function on $\bar D$.
(One can see that $W_{z_0}=h_{z_0}^{(0)}$ for $w=v$.)

\begin{lem}[\cite{NS}] \label{lem3}
For $w\in C_{\bar z}^1(\bar D)$ the following estimate holds:
\begin{equation} \label{estw}
|W_{z_0}(\lambda)|\le c_4(D) \frac{\log\,(3|\lambda|)}{ |\lambda|}  \|w\|_{C_{\bar z}^1(\bar D)},\ z_0\in\bar D,\ |\lambda|\ge 1.
\end{equation}
\end{lem}

\begin{lem} \label{lem4}
For $v \in C^1_{\bar z}(\bar D)$ and for $\|g_{z_0,\lambda} v \|^{op}_{C^1_{\bar z}(\bar D)} \leq \delta < 1$ we have that
\begin{align} \label{estmuk}
&\|\mu_{z_0}(\cdot, \lambda) - \mu_{z_0}^{(k)}(\cdot, \lambda)\|_{C^1_{\bar z}(\bar D)} \leq \frac{\delta^{k+1}}{1-\delta}, \\ \label{esth}
&|h_{z_0}(\lambda)-h^{(k)}_{z_0}(\lambda)| \leq c_5(D,n)\frac{\log(3|\lambda|)}{|\lambda|} \frac{\delta^{k+1}}{1-\delta} \|v\|_{C^1_{\bar z}(\bar D)},
\end{align}
where $z_0 \in D, \; \lambda \in \C,\; |\lambda |\ge 1, \; k \in \N \cup \{ 0 \}$.
\end{lem}

The proof of Lemma \ref{lem4} in the scalar case can be found in \cite{NS}: the generalization to the matrix-valued case is straightforward.

\begin{lem} \label{lem5}
The function $g_{z_0}(z,\zeta,\lambda)$ satisfies the following properties:
\begin{align} 
&g_{z_0}(z,\zeta,\lambda) \quad \textrm{is continuous for } z,\zeta \in \bar D, \;z \neq \zeta, \; z_0 \in D, \\ \label{estgg}
|&g_{z_0}(z,\zeta,\lambda)| \leq c_6(D) |\log |z- \zeta| |, \qquad z, \zeta \in \bar D, \; z_0 \in D,
\end{align}
where $\lambda \in \C$.
\end{lem}
These properties follow from the definition \eqref{defg} and from classical estimates (see \cite{V}).
\begin{lem} \label{lem6}
Under the assumptions of Proposition \ref{prop2}, the Schwartz kernel $(\Phi-\Phi_0)(z,\zeta)$ of the operator $\Phi-\Phi_0$ satisfies the following properties:
\begin{align}
&(\Phi-\Phi_0)(z,\zeta) \quad \textrm{is continuous for } z,\zeta \in \partial D, \;z \neq \zeta, \\
|&(\Phi-\Phi_0)(z,\zeta)| \leq c_7(D,v,n) |\log |z- \zeta| |, \qquad z, \zeta \in \partial D.
\end{align}
\end{lem}
For a proof of this Lemma in the scalar case we refer to \cite{N1,N2}: the generalization to the matrix-valued case is straightforward.

\section{Proofs of Theorem \ref{theo2}, Propositions \ref{prop2}, \ref{prop3} \\ and Corollary \ref{theo1}} \label{secpf}
We begin with a matrix version of Alessandrini's identity (see \cite{A} for the scalar case):
\begin{equation} \label{aless1}
\int_{\partial D} u_0(z) (\Phi - \Phi_0)u(z) |dz| = \int_D u_0(z) v(z) u(z) d \Ree z \, d \Imm z
\end{equation}
for any sufficiently regular $\matn$-valued function $u$ (resp. $u_0$) such that $\Delta u_0 = 0$ (resp. $(-\Delta +v) u = 0$) in $D$.
This follows from Stokes's theorem, exactly as in the scalar case.

The general matrix version of Alessandrini's identity (that will not be used)
\begin{equation} \label{aless2}
\int_{\partial D} u_1(z) (\Phi_2 - \Phi_1)u_2(z) |dz| = \int_D u_1(z) (v_2(z) -v_1(z)) u_2(z) d \Ree z \, d \Imm z
\end{equation}
for $u_1,u_2 \in C^2(\bar D, \matn)$ such that $(-\Delta + v_j) u_j =0$ in $D$, works if $u_1$ and $v_1$ commute each other (but does not work in general).

\begin{proof}[Proof of Theorem \ref{theo2}]
Let us begin with the proof of formulas \eqref{recv} and \eqref{estimv}: we have indeed
\begin{align} \label{convg}
\left|v(z_0) - \frac{2}{\pi} |\lambda| h_{z_0}(\lambda)\right| \leq \left|v(z_0) - \frac{2}{\pi}|\lambda| h^{(0)}_{z_0}(\lambda)\right|+\frac{2}{\pi}|\lambda||h_{z_0}(\lambda) - h^{(0)}_{z_0}(\lambda)|.
\end{align}
The first term in the right side goes to zero as $|\lambda| \to \infty$ by Lemma \ref{lem2}, while the other by Lemmata \ref{lem1} and \ref{lem4}. In addition, for $v \in C^2(\bar D,\matn)$ with $\|v\|_{C^2(\bar D)} <N_2$ and $\frac{\partial v}{\partial \nu}|_{\partial D} =0$, using \eqref{est1}, \eqref{estsolmu}, \eqref{estm} and \eqref{esth} we obtain, from \eqref{convg}:
\begin{align*}
\left|v(z_0) - \frac{2}{\pi} |\lambda| h_{z_0}(\lambda)\right| &\leq c_3(D,n)\frac{\log(3|\lambda|)}{|\lambda|}\|v\|_{C^2(\bar D)} \\
&\qquad + c_5(D,n)\frac{\log(3|\lambda|)}{|\lambda|^{1/2}}\|v\|^2_{C^1_{\bar z}(\bar D)} \\
&\leq c_8(D,n)\frac{\log(3|\lambda|)}{|\lambda|^{1/2}}(\|v\|_{C^2(\bar D)}+ \|v\|^2_{C^1_{\bar z}(\bar D)}),
\end{align*}
for $\lambda$ such that
\begin{align*} 
2n\frac{c_2(D)}{ |\lambda|^{\frac 1 2}}\|v\|_{C^1_ {\overline z}(\bar D)} \le \frac 1 2,\ |\lambda|\ge 1,
\end{align*}
which implies \eqref{esttv}. In order to prove \eqref{esttv2} we will need the following lemma: 

\begin{lem} \label{lem11}
Let $g_{z_0, \lambda} u$ be defined by \eqref{green}, where $u \in C^1_{\bar z}(\bar D)$, $z_0, \lambda \in \C$. Then the following estimate holds:
\begin{align}
\| &g_{z_0, \lambda} u \|_{C(\bar D)} \leq \eta(D) \frac{\log(3|\lambda|)}{|\lambda|^{\frac 3 4}} \|u \|_{C^1_{\bar z}(\bar D)}, \; |\lambda| \geq 1.
\end{align}
\end{lem}

\noindent
{\it Proof of Lemma \ref{lem11}.} As in the proof of \cite[Lemma 3.1]{NS}, we can write $g_{z_0, \lambda} = \frac 1 4 T \bar T_{z_0, \lambda}$, for $z_0, \lambda \in \C$, where
\begin{align*}
&T u(z) = - \frac{1}{\pi} \int_D \frac{u(\zeta)}{\zeta - z} d \Ree\zeta \, d \Imm \zeta, \\
&\bar T_{z_0, \lambda} u(z) = -\frac{e^{-\lambda(z- z_0)^2 + \bar \lambda(\bar z - \bar z_0)^2}}{\pi} \int_D \frac{e^{\lambda(\zeta-z_0)^2-\bar \lambda(\bar \zeta - \bar z_0)^2}}{\bar \zeta- \bar z} u(\zeta)  d \Ree\zeta \, d \Imm \zeta,
\end{align*}
for $z \in \bar D$ and $u$ a test function. We have that (see \cite{NS}):
\begin{align}
&Tw \in C^1_{\bar z}(\bar D), \\ \label{est32}
&\| Tw \|_{C^1_{\bar z}(\bar D)} \leq \eta_1(D) \| w \|_{C(\bar D)}, \; \textrm{where } w \in C(D), \\ \label{est33}
&\bar T_{z_0,\lambda} u \in C(\bar D), \\ \label{est34}
&\|\bar T_{z_0, \lambda} u\|_{C(\bar D)} \leq \frac{\eta_2(D)}{|\lambda|^{\frac 1 2}} \|u \|_{C^1_{\bar z}(\bar D)}, \; |\lambda| \geq 1, \\ \label{est35}
&\|\bar T_{z_0, \lambda} u\|_{C(\bar D)} \leq \frac{\log (3 |\lambda|)(1+|z-z_0|)\eta_3(D)}{|\lambda| |z- z_0|^2}\|u\|_{C^1_{\bar z}(\bar D)}, \; |\lambda| \geq 1,
\end{align}
where $u \in C^1_{\bar z}(\bar D)$, $z_0, \lambda \in \C$.

Let $z_0 \in D$, $0 < \delta < \frac 1 2$ and $B_{z_0,\delta} = \{ z \in \C : |z- z_0| < \delta \}$. We have
\begin{align} \label{estg}
|4\pi g_{z_0, \lambda} u(z)| &= \left| \int_D \frac{\bar T_{z_0, \lambda} u(\zeta)}{\zeta - z} d \Ree\zeta \, d \Imm \zeta \right|  \\ \nonumber
&\leq \int_{B_{z_0,\delta} \cap D} \frac{|\bar T_{z_0, \lambda} u(\zeta)|}{|\zeta - z|} d \Ree\zeta \, d \Imm \zeta + \int_{D \setminus B_{z_0,\delta}} \frac{|\bar T_{z_0, \lambda} u(\zeta)|}{|\zeta - z|} d \Ree\zeta \, d \Imm \zeta \\ \nonumber
&\leq 2 \pi \delta \frac{\eta_2(D)}{|\lambda|^{\frac 1 2}} \|u \|_{C^1_{\bar z}(\bar D)} + \frac{\log (3 |\lambda|)\eta_4(D)}{|\lambda| \delta}\|u\|_{C^1_{\bar z}(\bar D)},
\end{align}
where we used the following estimate:
\begin{align*}
&\int_{D \setminus B_{z_0,\delta}} \frac{1}{|\zeta - z||\zeta - z_0|^2} d \Ree\zeta \, d \Imm \zeta \\
&\qquad = \int_{B_{z,\delta} \cap (D \setminus B_{z_0,\delta})} \frac{1}{|\zeta - z||\zeta - z_0|^2} d \Ree\zeta \, d \Imm \zeta \\ \nonumber
&\qquad  \qquad + \int_{D \setminus ( B_{z,\delta} \cup B_{z_0,\delta})} \frac{1}{|\zeta - z||\zeta - z_0|^2} d \Ree\zeta \, d \Imm \zeta \\
&\qquad \leq \frac{2 \pi}{\delta} + \int_{D \setminus ( B_{z,\delta} \cup B_{z_0,\delta})} \frac{1}{|\zeta - z|^3}+ \frac{1}{|\zeta - z_0|^3} d \Ree\zeta \, d \Imm \zeta \\
&\qquad \leq \frac{\eta_5(D)}{\delta}.
\end{align*}
Putting $\delta =\frac{1}{2} |\lambda|^{-\frac{1}{4}}$ in \eqref{estg} we obtain the result. Thus Lemma \ref{lem11} is proved.\smallskip

We now come back to the proof of \eqref{esttv2}. Proceeding from \eqref{convg} and Lemma \ref{lem2} we obtain: 
\begin{align} \label{convg2}
\left|v(z_0) - \frac{2}{\pi} |\lambda| h_{z_0}(\lambda)\right| \leq c_3(D,n)\frac{\log(3|\lambda|)}{|\lambda|}\|v\|_{C^2(\bar D)}+\frac{2}{\pi}|\lambda||h_{z_0}(\lambda) - h^{(0)}_{z_0}(\lambda)|,
\end{align}
for $|\lambda| \geq 1$. In addition, from the definitions of $h^{(k)}\!, \; \mu^{(k)} \!$, Lemmata \ref{lem1} and \ref{lem4}, we have
\begin{align*}
&|h_{z_0}(\lambda) - h^{(0)}_{z_0}(\lambda)| \\
&\quad \leq \left|\int_D e^{\lambda (z-z_0)^2 -\bar \lambda (\bar z- \bar z_0)^2}v(z) g_{z_0,\lambda}v(z) d\Ree z \, d \Imm z \right| + O\left(\frac{\log(3|\lambda|)}{|\lambda|^2} \right)n^2 \|v\|^3_{C^1_{\overline z}(\bar D)},
\end{align*}
for $\lambda$ such that $2n\frac{c_2(D)}{ |\lambda|^{1/2}}\|v\|_{C^1_ {\overline z}(\bar D)} \le \frac 1 2,\ |\lambda|\ge 1$.

Repeating the proof of \cite[Lemma 3.3]{NS} and using also Lemma \ref{lem11}, we have, for $0<\varepsilon \leq 1$,
\begin{align} \label{ultima}
&\left|\int_D e^{\lambda (z-z_0)^2 -\bar \lambda (\bar z- \bar z_0)^2}v(z) g_{z_0,\lambda}v(z) d\Ree z \, d \Imm z \right| \\ \nonumber
&\leq \int_{D \cap B_{z_0,\varepsilon}}\! \! \! \! \! \! \! \!  \|v(z) g_{z_0,\lambda}v(z)\|_{C(\bar D)} d \Ree z \, d \Imm z +\frac{1}{4 |\lambda|} \int_{\partial (D \setminus B_{z_0,\varepsilon})} \! \! \! \! \! \! \! \! \! \! \! \! \! \! \! \! \! \frac{\|v(z) g_{z_0,\lambda}v(z)\|_{C(\bar D)}}{|\bar z - \bar z_0|}|dz| \\ \nonumber
&\qquad + \frac{1}{2|\lambda|}\int_{D \setminus B_{z_0,\varepsilon}}\left| \frac{\partial}{\partial \bar z}\left( \frac{v(z) g_{z_0,\lambda}v(z)}{\bar z-\bar z_0}\right) \right| d \Ree z \, d \Imm z 
 \\ \nonumber
&\leq  \sigma_1(D,n) \|v\|_{C(\bar D)} \|v\|_{C^1_{\overline z}(\bar D)} \frac{\varepsilon^2  \log(3 |\lambda|)}{|\lambda|^{3/4}} \\ \nonumber
&\qquad+ \sigma_2(D,n) \|v\|_{C(\bar D)} \|v\|_{C^1_{\overline z}(\bar D)} \frac{\log(3 \varepsilon^{-1}) \log(3 |\lambda|) }{|\lambda|^{1+3/4}} \\ \nonumber
&\qquad + \frac{1}{8|\lambda|} \left|\int_{D \setminus B_{z_0,\varepsilon}} e^{\lambda (z-z_0)^2 -\bar \lambda (\bar z- \bar z_0)^2} v(z) \frac{\bar T_{z_0, \lambda}v(z)}{\bar z- \bar z_0} d\Ree z \, d \Imm z \right|, \quad |\lambda| \geq 1,
\end{align}
where we also used integration by parts and the fact that $\frac{\partial}{\partial \bar z}g_{\lambda,z_0}u(z)= \frac{1}{4}\bar T_{z_0,\lambda}u(z)$. The last term in \eqref{ultima} can be estimated independently on $\varepsilon$ by
\begin{align}
\sigma_3(D,n) \frac{\log(3|\lambda|)}{|\lambda|^{1+3/4}} \|v\|_{C(\bar D)} \|v\|_{C^1_{\overline z}(\bar D)} 
\end{align}
using the same argument as in the proof of Lemma \ref{lem11} (see estimate \eqref{estg}). Now putting $\varepsilon = |\lambda|^{-1/2}$ in \eqref{ultima} we obtain
\begin{align*}
|\lambda||h_{z_0}(\lambda) - h^{(0)}_{z_0}(\lambda)| \leq \sigma_4(D,n)\frac{(\log(3|\lambda|))^2}{|\lambda|^{3/4}}\|v\|^2_{C^1_{\overline z}(\bar D)}( \|v\|_{C^1_{\overline z}(\bar D)} +1),
\end{align*}
for $|\lambda| > \rho_2(D,N_1,n),$ which, together with \eqref{convg2}, gives us \eqref{esttv2}. \smallskip

The proofs of the other formulas of Theorem \ref{theo2} are based on identity \eqref{aless1}. As $\mu_{z_0}(z,\lambda) = e^{-\lambda (z-z_0)^2} \psi_{z_0}(z,\lambda)$, we can write the generalized scattering amplitude as
\begin{equation} \nonumber
h_{z_0}(\lambda) = \int_D e^{-\bar \lambda (\bar z-\bar z_0)^2} v(z) \psi_{z_0}(z,\lambda) d \Ree z \, d \Imm z.
\end{equation}
Now identity \eqref{aless1} with $u_0(z) = e^{-\bar \lambda (\bar z-\bar z_0)^2}I$ and $u(z) = \psi_{z_0}(z,\lambda)$ reads
\begin{equation} \nonumber
\int_{\partial D} e^{-\bar \lambda (\bar z-\bar z_0)^2}(\Phi - \Phi_0)\psi_{z_0}(z,\lambda) |dz| = \int_D e^{-\bar \lambda (\bar z-\bar z_0)^2} v(z) \psi_{z_0}(z,\lambda) d \Ree z \, d \Imm z
\end{equation}
which gives formula \eqref{rech}.

Since $\mu_{z_0}$ is a solution of equation \eqref{defmu}, $\psi_{z_0}(z,\lambda)$ satisfies the equation
\begin{equation} \label{defpsi}
\psi_{z_0}(z,\lambda) = e^{\lambda(z-z_0)^2}I + \int_D G_{z_0}(z,\zeta,\lambda)v(\zeta) \psi_{z_0}(\zeta,\lambda) d \Ree \zeta \, d \Imm \zeta,
\end{equation}
for $z_0, z \in \bar D$, $\lambda \in \C$, $|\lambda| > \rho_1(D,N_1,n)$.
Thus again by identity \eqref{aless1}, with $u_0 = G_{z_0}(z,\zeta,\lambda)I$ and $u(z) = \psi_{z_0}(\zeta , \lambda)$, by \eqref{zetaaa} and \eqref{defpsi} we obtain, for $z \in \partial D$,
\begin{align*}
\int_{\partial D}\! \! \! G_{z_0}(z,\zeta, \lambda) (\Phi-\Phi_0) \psi_{z_0}(\zeta,\lambda) |d \zeta| &= \int_D G_{z_0}(z,\zeta,\lambda) v(\zeta) \psi_{z_0}(\zeta,\lambda) d \Ree \zeta \, d\Imm \zeta \\ \nonumber
&= \psi_{z_0}(z,\lambda) - e^{\lambda(z-z_0)^2}I.
\end{align*}
This finish the proof of Theorem \ref{theo2}.
\end{proof}

\begin{proof}[Proof of Proposition \ref{prop2}]
By \eqref{defG} we have that $G_{z_0}(z,\zeta, \lambda)$ satisfies the same properties as $g_{z_0}(z,\zeta, \lambda)$ in Lemma \ref{lem5}, with the difference that the constant in \eqref{estgg} depends also on $\lambda$. This observation, along with Lemma \ref{lem6}, implies that the operator $A(\lambda)$ defined as
\begin{equation} \nonumber
A(\lambda)u(z) = \int_{\partial D} G_{z_0} (z, \zeta, \lambda) (\Phi-\Phi_0)u(\zeta) |d \zeta|, \qquad z \in \partial D,
\end{equation}
for a test function $u$, is compact on the space of continuous functions on $\partial D$. Thus equation \eqref{bordpsi} is a Fredholm linear integral equation of the second kind in the space of continuous functions on $\partial D$. 
\end{proof}

\begin{proof}[Proof of Proposition \ref{prop3}]
First we have that equations \eqref{defmu} and \eqref{bordpsi} are well defined (i.e. Fredholm linear integral equations of the second type) on the spaces of continuous functions on $\bar D$ and $\partial D$ respectively. This follows from \eqref{estcontg} for the first equation and from Proposition \ref{prop2} for the second one. 
\smallskip

Now if \eqref{defmu} admits a solution $\mu_{z_0}(z,\lambda) \in C(\bar D)$, then by \eqref{contg} and \eqref{defmu} one readily obtains $\mu_{z_0}(z,\lambda) \in C^1_{\bar z} (\bar D)$. This solution is unique by Lemma \ref{lem1} for $|\lambda| > \rho_1(D,N_1,n)$ and by the same arguments as in the proof of Theorem \ref{theo2} one has that $\psi_{z_0}(z,\lambda)|_{z \in \partial D}$ satisfies equation \eqref{bordpsi}.
\smallskip

Conversely, suppose that $\psi_{z_0}(z,\lambda) \in C(\partial D)$ satisfies equation \eqref{bordpsi}: we have to show that $\psi_{z_0}(z,\lambda)$, defined on $\bar D$ as the solution of the Dirichlet problem $(-\Delta +v)\psi_{z_0}(z,\lambda) = 0$ with boundary values given by a solution of equation \eqref{bordpsi}, satisfies \eqref{defpsi}.

By identity \eqref{aless1}, $\psi_{z_0}(z,\lambda)$ satisfies already equation \eqref{defpsi} with $z \in \partial D$. Now, the function
\begin{equation}
\varphi(z) = \psi_{z_0}(z,\lambda)-e^{\lambda(z-z_0)^2}I -\int_D G_{z_0}(z,\zeta,\lambda)v(\zeta) \psi_{z_0}(\zeta,\lambda) d \Ree \zeta \, d \Imm \zeta
\end{equation}
satisfies $\Delta \varphi =0$ in $D$ and $\varphi|_{\partial D} = 0$, so $\varphi \equiv 0$ in $D$.
Proposition \ref{prop3} is proved.
\end{proof}

\begin{proof}[Proof of Corollary \ref{theo1}]
If $v_j|_{\partial D} = 0$, for $j=1,2$, then we can apply Theorem \ref{theo2} and Propositions \ref{prop2}, \ref{prop3}. As $\Phi_1 = \Phi_2$, then $\psi^1_{z_0}(\cdot,\lambda)|_{\partial D} = \psi^2_{z_0}(\cdot,\lambda)|_{\partial D}$ for $|\lambda| > \rho_1(D,N_1,n)$ (where we called $\psi^j_{z_0}(z,\lambda)$ the Bukhgeim analogues of the Faddeev solutions corresponding to $v_j$, for $j=1,2$). Thus we also have equality between the corresponding generalized scattering amplitudes, $h^1_{z_0}(\lambda) = h^2_{z_0}(\lambda)$ for $|\lambda| > \rho_1(D,N_1,n)$, which yields $v_1(z_0) = v_2(z_0)$ for $z_0 \in D$.

If $v_j|_{\partial D} \neq 0$, for $j=1,2$, and $D$ is such that \eqref{condbord} holds, then by Remark \ref{rem3} we can apply Theorem \ref{theo2} and argue as above.

The general case follows from stability estimates which will be published in another paper, following the scheme of \cite{NS}.
\end{proof}

\end{document}